\documentclass{article}
\usepackage{authblk}

\usepackage[english]{babel}
\usepackage{amssymb,amsthm,amsmath,mathrsfs}
\usepackage{graphicx}
\usepackage{xcolor}
\usepackage{enumerate}
\usepackage{cite}

\newtheorem{theorem}{Theorem}[section]
\newtheorem{proposition}[theorem]{Proposition}

\newtheorem{remark}[theorem]{Remark}
\newtheorem{problem}{Problem}

{\theoremstyle{definition}
\newtheorem*{definition*}{Definition}
}
\newtheorem*{proposition*}{Proposition}
\newtheorem*{corollary*}{Corollary}
\newtheorem*{lemma*}{Lemma}
\newtheorem*{remark*}{Remark}

\newcommand{\F}{\mathbb {F}}

\renewcommand{\P}{\mathbb {P}}
\renewcommand{\S}{\mathcal{S}}

\title{Rational points on cubic surfaces and AG codes from the Norm-Trace curve}
\author[1]{Matteo~Bonini
\thanks{The author was supported by the Irish Research Council, grant n. GOIPD/2020/597.}}
\author[2]{Massimiliano Sala}
\author[3]{Lara Vicino}
\affil[1]{School of Mathematics and Statistics, University College Dublin, Dublin, Ireland e-mail: matteo.bonini@ucd.ie}
\affil[2]{Department of Mathematics, University of Trento, Trento, Italy e-mail: massimiliano.sala@unitn.it}
\affil[3]{Department of Applied Mathematics and Computer Science, Technical University of Denmark, Kgs. Lyngby, Denmark e-mail: lavi@dtu.dk}
\date{}

\begin{document}

\maketitle

\begin{abstract}
In this paper we give a complete characterization of the intersections between the Norm-Trace curve over $\F_{q^3}$ and the curves of the form $y=ax^3+bx^2+cx+d$, generalizing a previous result by Bonini and Sala, providing more detailed information about the weight spectrum of one-point AG codes arising from such curve. We also derive, with explicit computations, some general bounds for the number of rational points on a cubic surface defined over $\F_{q}$. 
\end{abstract}



Many of the best performing algebraic codes are known to be Algebraic Geometry (AG) codes, which arise from algebraic varieties over finite fields. Among them, the most studied codes are the ones arising from algebraic curves, that were introduced by Goppa in the '80s; see \cite{goppa1981codes,Goppa82} for a detailed description. 

Let $\mathcal{X}$ be an algebraic curve defined over the finite field with $q$ elements $\mathbb{F}_q$. The parameters of codes arising from $\mathcal{X}$ strictly depend on some geometrical properties of the curve. In general, curves with many $\mathbb{F}_q$-rational places with respect to their genus give rise to long AG codes with good parameters. For this reason, maximal curves, that is, curves attaining the Hasse-Weil upper bound, have been widely investigated in the literature; see  \cite{Stichtenoth1988,Tiersma1987,XL2000,BMZ1,M2004,BalB}.

In general, the determination of the weight spectrum of a code $C$ (i.e. the set of the possible weights of $C$) is a very difficult task, see for instance \cite{Klove}.

For AG codes, it is possible to derive information about their weight spectrum by the study of the intersection of the base curve $\mathcal{X}$ and low degree curves, as done in  \cite{ballico2013duals,BB2018,C2012,MPS2014,MPS2016}.

The Norm-Trace curves are a natural generalization of the celebrated Hermitian curve to any extension field $\F_{q^r}$, and their codes have been widely studied; see \cite{geil2003codes,BS2018,ballico2013duals,FMT2013,MTT2008}.

In this paper, we focus on the intersection between the Norm-Trace curve over $\F_{q^3}$ and curves of the form $y = Ax^3 + Bx^2 + Cx + D$, giving the proof of a result which corrects a previous conjecture stated in \cite{BS2018}. In addition to this, we partially deduce the weight spectrum of its one-point codes in the place at the infinity.

In order to obtain these results, we translate the problem of finding the planar intersection between the cubic Norm-Trace curve and the above mentioned curves into that of counting the number of $\F_q$-rational points of certain cubic surfaces, which is a well-knwon topic in algebraic geometry in positive characteristic, see \cite{coray:singular,manin:cubic,weil1958abstract}.
For this reason, in section \ref{sec:general} we start by presenting general results on the number of $\F_q$-rational points on cubic surfaces defined over $\F_q$. These results end up constituting a partial generalization of a classical result by Weil and we obtain them exploiting classical results on cubic surfaces over finite fields.

In sections \ref{2cubic}, \ref{B=C=0} and \ref{sec:isosing:general} we analyze cubic surfaces arising from the intersection between the Norm-Trace curve and curves of the form $y = Ax^3 + Bx^2 + Cx + D$ over $\F_{q^3}$, proving Theorem \ref{thm:goal}. 

Finally, in section \ref{sec:agcodes}, we use these results to investigate the weight spectrum of a certain family of one-point codes arising from the cubic Norm-Trace curve. 

\section{Preliminaries}
Let $q=p^h$, where $p$ is a prime and $h>0$ an integer, and denote with $\F_q$ the finite field with $q$ elements. 

We recall that the \textit{norm} $\text{N}_{\F_q}^{\F_{q^r}}$ and the \textit{trace} $\text{T}_{\F_q}^{\F_{q^r}}$, where $r$ is a positive integer, are functions from $\F_{q^r}$ to $\F_q$ such that
\[
\text{N}_{\F_q}^{\F_{q^r}}(x)=x^{\frac{q^r-1}{q-1}}=x^{q^{r-1}+q^{r-2}+\dots+q+1}
\]
and
\[
\text{T}_{\F_q}^{\F_{q^r}}(x)=x^{q^{r-1}}+x^{q^{r-2}}+\dots+x^q+x.
\]
When $q$ and $r$ can be derived unequivocally from the context, we will omit the subscripts.

\subsection{The Norm-Trace curve}
The Norm--Trace curve $\mathcal{N}_r$ is the curve defined over the affine plane $\mathbb{A}^2(\F_{q^r})$ by the equation 
\begin{equation}
\label{eq:NormTrace}
\mathrm{N}(x)=\mathrm{T}(y).
\end{equation}

If $r=2$ the curve $\mathcal{N}_{r}$ is smooth, while if $r\geq 3$ it can be easily seen that $\mathcal{N}_{r}$ has a singular point which is the point at the infinity $P_{\infty}$. It is then well-known that $\mathcal{N}_r$ has $q^{2r-1}$ affine places and a single place at the infinity; in fact, there is exactly one place centered at each affine point of $\mathcal{N}_r$ (these are all smooth points) and the point at the infinity is either smooth (in the case $r=2$) and hence center of exactly one place, or it is singular and center of only one branch of the curve (if $r\geq 3$), so that $\mathcal{N}_{r}$ has a unique place at the infinity also in this case.

If $r=2$, $\mathcal{N}_r$ coincides with the Hermitian curve, and this is the only case in which $\mathcal{N}_r$ is smooth, since, as noted above, for $r\ge3$ it has a singularity in $P_{\infty}$.

Moreover it is known that its Weierstrass semigroup in the place centered at $P_{\infty}$ is generated by $\left\langle q^{r-1}, \frac{q^r-1}{q-1}\right\rangle$, see \cite{geil2003codes}. Also, the automorphism group is determined by the following result.
\begin{theorem}[\cite{BBZ2020}]
The automorphism group of $\mathcal{N}_r$ $\mathrm{Aut}(\mathcal{N}_{r})$ has order $q^{r-1}(q^r-1)$ and is a semidirect product $G\rtimes C$, where
\[
G=\left\{ (x,y)\mapsto(x,y+a)\mid \mathrm{T}(a)=0 \right\}
\] 
\[
C=\{(x,y)\mapsto(b x,b^{\frac{q^r-1}{q-1}}y)\mid b\in\mathbb F_{q^r}^*\}.
\]
\end{theorem}

Our main aim is the study of the planar intersection (i.e. the intersection counted without multiplicity) between $\mathcal{N}_r$ and the cubic curves of the form $y=ax^3+bx^2+cx+d$, where $a,b,c,d\in\F_{q^r}$ and $a\ne0$, in the case $r=3$. 
The case $r=2$ and $a=b=0$ was investigated in \cite{ballico2014} and the case $r=2$ and $a=0$ was completely investigated in \cite{MPS2014,donati2009intersection}. On the other hand, the case $r=3$ and $a=0$ was investigated in \cite{BS2018}.  

In \cite{BS2018}, the authors claim the following result.
\begin{theorem}
	The number of planar intersections between $\mathcal{N}_3$ and cubic curves of the form $y=ax^3+bx^2+cx+d$, where $a,b,c,d\in\F_{q^3}$, $a\neq 0$, is bounded by $q^2+7q+1$, when the surface defined over $\F_{q^3}$ by the equation 
	\[
    X_0X_1X_2=AX_0^3+A^qX_1^3+A^{q^2}X_2^3+BX_0^2+B^qX_1^2+B^{q^2}X_2^2+CX_0+C^qX_1+C^{q^2}X_2+E
	\]
	is irreducible, for $A,B,C,D,E\in\F_{q^3}$.
\end{theorem}

Still, a proof of this result is not given in the paper.  The aim of this paper is to give a detailed proof of this bound in any possible case, showing that this is not correct when such surface is a cone over an elliptic curve.

\subsection{Algebraic Geometry codes}
	
	In this section we recall some basic facts on AG codes. For a detailed discussion we refer to \cite{stich:agcodes}.
	
	Let $\mathcal{X}$ be a projective curve over the finite field $\F_q$ and  consider the function field $\F_q(\mathcal{X})$ of rational functions defined over $\F_q$. Denote with $\mathcal{X}(\F_q)$ the set of the $\F_q$-rationals points of $\mathcal{X}$.
	A divisor $D$ on $\mathcal{X}$ can be seen as a finite sum $\sum_{P\in\mathcal{X}(\overline{\F}_q)}n_PP$, where the $n_P$s are integers.  For a function $f \in \mathbb{F}_q(\mathcal{X})$, $(f)$ denotes the divisor associated to $f$. 
	A divisor $D$ is $\mathbb{F}_q$-rational if it coincides with its image $\sum_{P\in\mathcal{X}(\overline{\F}_q)}n_P^qP^q$ under the Frobenius map over $\F_q$. 
	
	Given an $\F_q$-rational divisor $D=\sum_{i=1}^n n_iP_i$ on $\mathcal{X}$, its support is defined as $\mathrm{supp}(D)=\{P_i : n_i \neq 0\}$.  
	
	The Riemann-Roch space associated with $D$ is the $\mathbb{F}_q$-vector space 
	\[
	\mathscr{L}(D)=\{f\in\F_{q}(\mathcal{X}) \ |\ (f)+D\ge0\}\cup \{0\}.
	\]
	It is known that $\mathscr{L}(D)$  is an $\F_q$-vector space of finite dimension $\ell(D)$. The exact dimension of this space can be computed using Riemann-Roch theorem.
	
	Consider now the divisor $D=\sum_{i=1}^nP_i$ where all the $P_i$'s are $\F_q$-rational. Let $G$ be another $\F_q$-rational divisor on $\mathcal{X}$ such that $supp(G)\cap supp(D)=\emptyset$. Consider the evaluation map 
	\[
	\begin{split}
	e_D:\quad \mathscr{L}(G)&\longrightarrow \F_q^n\\
	f\quad &\longmapsto e_D(f)=(f(P_1),\dots,f(P_n)).
	\end{split}
	\]
	The map $e_D$ is $\F_q$-linear and it is injective if $n>\deg(G)$.
	
	The AG code $C_{\mathscr{L}}(D,G)$, also called functional code, associated with the divisors $D$ and $G$ is defined as $C_{\mathscr{L}}(D,G):=e_D(\mathscr{L}(G))=\{(f(P_1),\ldots,f(P_n)) \ | \ f\in \mathscr{L}(G)\}\subseteq \F_q^n$. Such a code is an $[n,\ell(G)-\ell(G-D),d]_q$ code, where $d\ge \bar{d}=n-\deg(G)$ and $\bar{d}$ is the so-called designed minimum distance (of such code).

\section{General results on the number of rational points on cubic surfaces}
\label{sec:general}

In this section we will exploit well-known results on cubic surfaces over finite fields to explicitly obtain some general results on the number of rational points on a cubic surface $\S$ defined over a finite field $\F_q$. In addition to being theoretically interesting, these results will be useful later on in the paper; in fact, we are going to study rational points on particular cubic surfaces in order to bound the intersections between the norm-trace curve over $\F_{q^3}$ and the curves of the form $y=ax^3+bx^2+cx+d$.

Our final goal consists in obtaining a bound in the form 
$$|\S(\F_q)| \leq q^2+7q+1$$
for such surfaces, since this has interesting impact on the weight spectrum of the induced AG codes. We will obtain this result for a specific family of cubic surfaces. However, in this section we show that the bound stated in Theorem \ref{thm:goal} holds more in general, being a partial generalization of the following result due to Weil.

\begin{theorem}[\cite{manin:cubic}, Theorem 23.1]
		\label{thm:weil}
		Let $\mathcal{S}$ be a smooth irreducible cubic surface over $\F_q$, then the number of points of $\mathcal{S}(\F_q)$ is exactly
		\[
		|\mathcal{S}(\F_q)|= q^2+\eta q+1
		\]
		where $\eta\in\{-2,-1,0,1,2,3,4,5,7\}$.
\end{theorem}

In fact, we are going to prove the following result.

\begin{theorem}
\label{thm:goalgeneral}
Consider an $\F_q$-rational cubic surface $\S$ and assume it is not a cone over a smooth irreducible cubic plane curve. If $\S$ satisfies one of the following conditions
\begin{itemize}
    \item[(i)] $\S$ is absolutely irreducible with non-isolated singularities,
    \item[(ii)] $\S$ is absolutely irreducible with only isolated singularities, one of which is an $\F_q$-rational point,
    \item[(iii)] $\S$ is absolutely irreducible with only two isolated singularities,
\end{itemize}
then
		\[
		|\S(\F_q)|\le q^2+7q+1.
		\]
\end{theorem}

For brevity, in what follows we will sometimes just write irreducible instead of absolutely irreducible, since we will see all the surfaces as elements of $\mathbb{A}^3(\overline{\F}_q)$.

Note that, if $\S$ is irreducible, it cannot be a cone over a reducible plane curve, otherwise it would be reducible as well. 

We will discuss separately (at the end of the section) the cases in which $\S$ is a cone over a smooth absolutely irreducible cubic plane curve or $\S$ is reducible, which are not comprised in Theorem \ref{thm:goalgeneral}; we assume henceforth that $S$ is not a cone over a smooth absolutely irreducible cubic plane curve and consider only the irreducible case. We will also make some observations on the cases in which $\S$ has only three or four isolated singular points.

\subsection{$\S$ irreducible with non-isolated singularities}
\label{sec:non-isolated}
In this subsection we analyze the case in which $\S$ has non-isolated singularities. Note that, if $\S$ has only a finite number of singular points, then these would be necessarily isolated. Therefore, in this case $\S$ has an infinite number of singular points.

We treat separately the case in which the surface is a cone or not. The possibilities (see \cite[Chapter 2, Section 4]{kaplan2013rational} and \cite{roczen1996cubic}) are the following:

\begin{enumerate}[(i)]
\item $\S$ is a cone over a singular absolutely irreducible cubic plane curve
\item $\S$ is not a cone.
\end{enumerate}

\subsubsection{$\S$ is a cone over a singular cubic curve}
\begin{proposition}
Let $\S$ be an irreducible cubic surface defined over $\F_q$. If $\S$ is a cone over a singular absolutely irreducible cubic plane curve, then \begin{equation}
\label{bound:coneoversingularcubic}
|\S (\F_q)|\leq (q+2)q+1=q^2+2q+1.
\end{equation}
\end{proposition}

\begin{proof}
Take $P$ the vertex of the cone and choose an arbitrary plane in $\P^3(\F_q)$ not containing $P$. We can think of the cone as the union of the lines between $P$ and a singular absolutely irreducible cubic curve lying on this plane. Hence, if the cubic curve has $m$ $\F_q$-rational points, the obtained cone has $mq+1$ $\F_q$-rational points (see \cite[Chapter 2, Section 4]{kaplan2013rational}). 

This is easy to see as $\F_q$-rational points of the cone lie on lines through $P$ and $\F_q$-rational points of the curve, so we count $q$ point for every line (all $\F_q$-rational points except the vertex $P$) and finally we add $P$. 

Then, due to well-known results, see \cite[Chapter 11]{hirsch:projective} for details, a singular absolutely irreducible cubic curve in $\P^2(\F_q)$ can have at most $q+2$ $\F_q$-rational points, and this yields the thesis.

\end{proof}

\begin{remark}
Note that, counting points as in the previous proposition, we are actually counting projective $\F_q$-rational points, not only the affine ones, so the bound we obtain is not tight; nonetheless, it is compatible with our goal bound which is $q^2+7q+1$.
\end{remark}

\subsubsection{$\S$ is not a cone}

If $\S$ is not a cone and has no isolated singularities, it turns out that Theorem \ref{thm:goalgeneral} is established by results in \cite[Chapter 2, Section 4]{kaplan2013rational} which exploit a theorem of Weil (see also \cite{elkies2006linear}, \cite{weil1958abstract} and \cite[Chapter IV]{manin:cubic} for details).

Weil's theorem (see \cite[Theorem 23]{kaplan2013rational}) gives a suitable bound for our purposes since it ensures that the surface $\S$ has necessarily a number of $\F_q$-rational points that is

\begin{equation}
\label{bound:nonisolatedsing}
q^2+q+1+tq
\end{equation} 

for some $t\in [-3,6]$.

\begin{theorem}[Weil]
Let $\S$ be a surface defined over a finite field $\F_q$. If $S \otimes \overline{\F}_q$ is birationally trivial, then

$$
|S(\F_q)|=q^2+q\mathrm{Tr}(\varphi^{*})+1,
$$
where $\varphi$ denotes the Frobenius endomorphism and $\mathrm{Tr}(\varphi^{*})$ denotes the trace of $\varphi$ in the representation of $\mathrm{Gal}(\overline{\F}_q/\F_q)$ on $\mathrm{Pic}(S \otimes \overline{\F}_q)$. 
\end{theorem}

As noted in \cite[Chapter 2, Section 4]{kaplan2013rational}, this theorem is useful in our case because the surface $\S$ is the anti-canonical model of a degree 3 weak Del Pezzo surface.

As observed, to prove the desired result in this case we need to know that the surface has no isolated singularities. This is established from the classification of cubic surfaces (see \cite{roczen1996cubic} and \cite[proof of Proposition 6.4]{hartshorne1997families}), from which it is actually known even more on the singular locus of such a surface; we remark this in the following proposition.

\begin{proposition}
\label{prop:singline}
Let $\mathcal{S}\subseteq \mathbb{P}^{3}(\mathbb{K})$ be an irreducible cubic surface with non-isolated singularities over a field $\mathbb{K}$ of positive characteristic. Then all the singular points of $\mathcal{S}$ lie on a double line $\ell$. 
\end{proposition}

\begin{proof}
We start by observing that, by a result due to Bertini and proved in positive characteristic by Nakai in \cite{nakai1950note}, a general hyperplane $L$ in $\mathbb{P}^{3}(\mathbb{K})$ cuts on $\mathcal{S}$ an irreducible cubic curve $\mathcal{C}$ whose singular points are exactly the singular points of $\mathcal{S}$ lying on the hyperplane $L$.\\
Since the curve $\mathcal{C}$ is an irreducible cubic curve, it cannot have more than one singular double point, which means that on a general hyperplane must lie at most one singular point of the surface $\mathcal{S}$ (see also \cite[Section 2]{bruce1979classification}). This leads to conclude that  the singular set of $\mathcal{S}$ is a line $\ell$, which is double since $\mathcal{S}$ is a cubic irreducible surface. Indeed, if $\ell$ were a line of multiplicity higher than two, then $\mathcal{S}$ would have to be the union of three (distinct or coincidental) planes through the line $\ell$ (see \cite{abhyankar1960cubic}).

If there were a singular point $Q$ of $\mathcal{S}$ not lying on $\ell$, then the plane through $Q$ and $\ell$ would be a component of $\mathcal{S}$, which is impossible since the surface is irreducible. Hence, there are no singular points of $\mathcal{S}$ not lying on $\ell$ and we have the thesis.       
\end{proof}

\subsection{$\S$ irreducible with isolated singularities}
\label{subsec:isolatedsing}
We investigate now the case in which $\S$ has only isolated singularities.
We first recall that, considering the classification of cubic surfaces in positive characteristic, if $\S$ is an irreducible cubic surface with isolated singularities, which is not a cone over a smooth plane cubic curve, then its singular points are double points, as it is noted in \cite{roczen1996cubic}.

Therefore, we can exploit the following useful theorem:

\begin{theorem}[\cite{coray:singular}]
\label{thm:4singular}
Let $\mathcal{S}\subset\mathbb{P}^3(\mathbb{K})$ be a singular irreducible cubic surface defined on the field $\mathbb{K}$. Let $\bar{\S}=\S(\overline{\mathbb{K}})$ be the surface defined by $\S$ over $\overline{\mathbb{K}}$, the algebraic closure of $\mathbb{K}$. Let $\delta$ be the number of isolated double points of $\bar{\S}$. Then $\delta\le 4$ and $\mathcal{S}$ is birationally equivalent (over $\mathbb{K}$) to
		\begin{enumerate}[(i)]
			\item $\P^2(\mathbb{K})$ if $\delta=1,4$;
			\item a smooth Del Pezzo surface of degree 4 if $\delta=2$;
			\item a smooth Del Pezzo surface of degree 6 if $\delta=3$.
		\end{enumerate}
	\end{theorem}

We recall that a smooth Del Pezzo surface is a smooth projective surface $\mathcal{X}$ whose anticanonical class is ample.

In this section, we just focus on the cases in which $\S$ has one singular $\F_q$-rational point or it has two singular points, as these are the cases comprised in Theorem \ref{thm:goalgeneral}.

We recall the following important fact (\cite{BS2018}).

\begin{remark}
\label{rem:rational}
As $\mathcal{S}$ is defined over $\F_q$, if $P\in{\mathcal{S}}(\F_q)$ is a singular point then its conjugates with respect to the Frobenius automorphism are also singular.
\end{remark}

\subsubsection{One singular $\F_{q}$-rational point}
\label{subsec:1sing}
In this subsection we investigate the case in which $\mathcal{S}$ has one singular $\F_q$-rational point. As previously noted (see Remark \ref{rem:rational}), if there is only one singular point $P$ on $\S$, then $P$ is necessarily $\F_{q}$-rational, otherwise its conjugates would be singular points of the surface. 

In case $\S$ has at least one $\F_q$-rational point (which is clearly verified when $\S$ has only one singularity), we can use the following idea to find the desired bound for the number of $\F_{q}$-rational points of $\S$. Note that this reasoning makes use of similar arguments with respect to the secant and tangent process to generate new rational points from old ones, see \cite{segre:ts,manin:cubic,Cooley2014} for a deeper investigation.

Consider the sheaf of $\F_{q}$-rational lines through the $\F_q$-rational point $P$. 

Since $P$ is a double point of the surface, a line $\ell\not\subset\S$ through $P$ can intersect $\S$ in at most one more point, which has to be $\F_{q}$-rational by construction, since we are intersecting two $\F_q$-rational varieties. 
Indeed, the coordinates of the points of intersection correspond to the roots of a cubic equation with coefficients in $\F_{q}$, and this equation has an $\F_{q}$-rational double root, corresponding to $P$. For this reason, also its third root has to be $\F_{q}$-rational.

On the one hand, a direct bound for the number of $\F_{q}$-rational lines through $P$ in the affine space is $q^{2}$. On the other hand, an $\F_{q}$-rational line is a subspace of dimension 1 of $(\F_{q})^{3}$, so it has exactly $q$ points in affine space. Therefore, an $\F_{q}$-rational line $\ell\subset\S$ through $P$ (i.e. $\ell\subset\S(\F_{q})$) has at most $q$ affine points on $\S$. This means that it has, apart from $P$, at most other $q-1$ $\F_{q}$-rational points on the surface.

These reasonings lead to the following bound

\begin{equation}
\label{bound:1singular}
|\S(\F_{q})|\leq q^{2} + h(q-1)
\end{equation}
where $h$ is the number of $\F_q$-rational lines through $P$ contained in $\S$.

\begin{remark}
The argument we have exploited makes possible to find all $\F_{q}$-rational points on the surface $\S$ because every $\F_{q}$-rational point on the surface lies on an $\F_{q}$-rational line through $P$. In fact, the line through $P$ and the $\F_q$-rational point considered is an $\F_q$-rational line ($P$ is indeed $\F_q$-rational). 
\end{remark}

Our aim now is to determine a bound for the value of $h$. For this purpose, we consider a result proved in \cite{datta:sixlines}.

\begin{theorem}[\cite{datta:sixlines}]
\label{thm:sixlines}
Let $Y\subset \P^{3}$ be a surface of degree $d$ defined over $\F_{q}$ and $P\in Y(\F_{q})$. Then one of the following holds:
\begin{itemize}
\item[$(a)$] $Y$ contains a plane defined over $\F_{q}$,
\item[$(b)$] $Y$ contains a cone over a plane curve defined over $\F_{q}$ with center at $P$,
\item[$(c)$] $|\{l\subset \P^{3} \ | \ l \ \mbox{is a line such that} \ P\in  l\subset Y\}|\leq d(d-1)$.
\end{itemize}
\end{theorem}

Note that this result is very useful for our purposes, since $\S$ is by hypothesis absolutely irreducible and not a cone over a smooth absolutely irreducible cubic plane curve. Moreover, by the classification of cubic surfaces in positive characteristic (see \cite{roczen1996cubic}), we also know that $\S$ cannot contain a cone over a smooth absolutely irreducible cubic plane curve. Since we are supposing that $\S$ has only isolated singularities, it cannot even contain a cone over a non-singular curve, otherwise it would have a double line.
Specializing this theorem to $\S$ and $P$, and recalling that $\S$ is defined by a homogeneous equation of degree 3 with coefficients in $\F_{q}$, we get that 

\[h\leq |\{l\subset \P^{3} \ | \ l \ \mbox{is a line such that} \ P\in  l\subset \S\}|\leq 6\]

that, combined with \eqref{bound:1singular}, gives the following result.

\begin{proposition}
If $\S$ has (at least) one singular $\F_{q}$-rational point then 
\begin{equation}
\label{bound:1singular:6}
|\S_{1}(\F_{q})|\leq q^{2} + 6(q-1)= q^{2} +6q -6.
\end{equation}
\end{proposition}

\subsubsection{Two singular points}
\label{subsec:2sing}
We consider now the case in which $\S$ has exactly two singular points.

Note that, by Remark \ref{rem:rational}, if $P_1$ and $P_2$ are the singular points of $\S$, then either $P_1$ and $P_2$ are both $\F_q$-rational, and in this case we already have the bound determined in the previous subsection \ref{subsec:1sing}, or $P_1$ and $P_2$ are $\F_{q^2}$-rational and conjugates.

We have to investigate this second case, for which turns out that we can adopt the same argument used in \cite{BS2018}, that comes from the investigation of the possible curves obtained from the intersection between $\S$ and the sheaf of planes through the line between the singularities. 

	\begin{proposition}[\cite{BS2018}]
		If $\S$ has two singular $\F_{q^2}$-rational conjugate points then 
		\begin{equation}
		\label{bound:2singular}
		q^2-14q+39\le |\S_1(\F_q)|\le q^2-q.
		\end{equation}
	\end{proposition}

\begin{remark}
\label{rem:cases1}
The method used for treating this case actually does not rely on the surface $\S$ having exactly two singular points, but just on $\S$ having at least two singular $\F_{q^2}$-rational points. Hence, this case solves also the cases in which the surface has three or four singular points, two of which are $\F_{q^2}$-rational and conjugates.
\end{remark}

\begin{remark}
\label{rem:cases2}
Note that, according to Theorem \ref{thm:4singular} and Remark \ref{rem:rational}, for $\S$ with isolated singularities there are still two other possibilities which are not covered by the results presented so far (and which we have not taken into account here since they are not comprised in Theorem \ref{thm:goalgeneral}). These are 
\begin{itemize}
    \item $\S$ having three $\F_{q^3}$-rational and conjugates singular points;
    \item $\S$ having four $\F_{q^4}$-rational and conjugates singular points.
\end{itemize}
We will discuss these cases in detail in section \ref{sec:isosing:general}, for the special class of cubic surfaces defined in section \ref{2cubic}. However, a trivial bound for these cases is $|\S(\F_q)|\leq 3q^2$. 
\end{remark}

\subsection{Missing cases}

In this final part of the section, we discuss the two remaining cases not comprised in Theorem \ref{thm:goalgeneral}, i.e., $\S$ being a cone over a smooth irreducible cubic plane curve or $\S$ being reducible. We investigate these two cases for completeness, and for the future applications to AG codes.

If the $\F_q$-rational surface $\S$ is a cone over a smooth absolutely irreducible cubic plane curve $\mathcal{E}$, as seen in section \ref{sec:non-isolated}, we have that the $\F_q$-rational points on $\S$ are $mq+1$, where $m$ is the number of $\F_q$-rational points of $\mathcal{E}$.

Since the cubic curve in this case is a smooth irreducible cubic plane curve, by well-known results in \cite[Chapter 11]{hirsch:projective} we know that it has at most $q+2\sqrt{q}+1$ $\F_q$-rational points in $\P^2(\F_q)$; this yields the following result.

\begin{proposition}
Let $\S$ be an $\F_q$-rational cubic surface. If $\S$ is a cone over a smooth absolutely irreducible cubic plane curve $\mathcal{E}$, then 

\begin{equation}
\label{bound:coneoversmoothcubic}
|\S (\F_q)|\leq (q+2\sqrt{q})q+1=q^2+2q\sqrt{q}+1.
\end{equation} 

\end{proposition}

Finally, if $\S$ is reducible, then three possible situations can happen:
\begin{enumerate}[(i)]
\item $\S$ is the union of a non-singular quadric surface and a plane;
\item $\S$ is the union of a quadric cone and a plane;
\item $\S$ is the union of three planes.
\end{enumerate}

In the worst possible case, which corresponds to the case of a complete reducibility of the surface in three $\F_q$-rational planes,
we obtain the trivial bound 
\[
|\S(\F_q)|\leq 3q^2.
\]
For the other cases this bound still holds, but it is possible to refine it, for instance considering well-known results on points on quadric surfaces which can be found in \cite[Chapter 15]{hirsch:finite}.

\section{Cubic surfaces from intersections of algebraic curves}
\label{2cubic}

We start now to investigate the intersection over $\mathbb{F}_{q^{3}}$ of the Norm--Trace curve $\mathcal{N}_{3}$ with the curve defined by 
$$y=\mathcal{A}(x)=A_{3}x^3+A_{2}x^{2}+A_{1}x+A_{0}$$
where $A_{3}\not = 0$ and $A_{i}\in \mathbb{F}_{q^{3}}$ for $i=0,1,2,3$.

As already recalled, when we refer to a planar intersection (or simply intersection) of two curves lying in the affine space $\mathbb{A}^{2}(\F_{q^{3}})$, we mean the number of points in $\mathbb{A}^{2}(\F_{q^{3}})$ lying in both curves, disregarding multiplicity (see \cite{BS2018}).
For the remaining part of this section, we exploit the same approach used in \cite{BS2018} to set the problem.

From now on, we will write $\mathrm{N}$ and $\mathrm{T}$ instead of $\mathrm{N}_{\F_q}^{\F_{q^3}}$  and $\mathrm{T}_{\F_q}^{\F_{q^3}}$, respectively. Moreover, throughout the paper we will always consider the curves (resp. surfaces) in the algebraic closure $\overline{\F}_q$ of $\F_q$, even when not stated explicitly. When we will need to consider smaller fields, we will point it out in the tractation.

Substituting $y=\mathcal{A}(x)$ in the equation of $\mathcal{N}_{3}$, and exploiting the linearity of the trace function, we get

\begin{equation}
\label{normtrace:y=A(x)}
\mathrm{N}(x)=\mathrm{T}(A_{3}x^3)+\mathrm{T}(A_{2}x^{2})+\mathrm{T}(A_{1}x)+\mathrm{T}(A_{0})
\end{equation}

Consider now a normal basis $\mathcal{B}=\{\alpha,\alpha^q,\alpha^{q^{2}}\}$, for a suitable $\alpha\in\F_{q^3}$. We know that such a basis exists, see \cite[Theorem 2.35]{lidl:ff}. The vector space isomorphism 

$$\Phi_{\mathcal{B}}:(\F_{q})^{3} \longrightarrow \F_{q^{3}}$$
$$\Phi_{\mathcal{B}}((s_{0},s_{1},s_{2}))=s_{0}\alpha_{0}+s_{1}\alpha^q+s_{2}\alpha^{q^2}$$
allows us to read the norm and the trace as maps from $(\F_{q})^{3}$ to $\F_{q}$, considering $\widetilde{\text{N}}=\text{N}\circ\Phi_{\mathcal{B}}$ and $\widetilde{\text{T}}=\text{T}\circ\Phi_{\mathcal{B}}$.
Let $\text{T}_i:=\text{\text{T}}(A_ix^i)$  and $\widetilde{\text{T}}_i:=\text{T}_i\circ\Phi_{\mathcal{B}}$, for $1\le i\le 3$, then it is readily seen that $\widetilde{\text{N}}$ and $\widetilde{\text{T}}_i$ are homogeneous polynomials of degree respectively $3$ and $i$ in $\F_q[x_0,x_1,x_2]$, $i=0,1,2,3$.

Therefore, we can rewrite \eqref{normtrace:y=A(x)} as
	\begin{equation}
	\label{eq:sup}
	\widetilde{\text{N}}(x_0,x_1,x_2)=\widetilde{\text{T}}_3(x_0,x_1,x_2)+\widetilde{\text{T}}_2(x_0,x_1,x_2)+\widetilde{\text{T}}_1(x_0,x_1,x_2)+E
	\end{equation}
where $E=\text{T}(A_0)$.
Equation \eqref{eq:sup} is the equation of a hypersurface of $\mathbb{A}^3(\overline{\F}_q)$, where $\overline{\F}_q$ is the algebraic closure of $\F_q$, and both RHS and LHS have degree 3.

All the considerations made so far lead to the fact that $\Phi_{\mathcal{B}}^{-1}$ induces a correspondence between $\F_{q^{3}}[x]$ and $\F_{q}[x_0,x_1,x_2]$, allowing us to substitute $x$ with $x_0\alpha+x_1\alpha^{q}+x_2\alpha^{q^{2}}$. 

We exploit this relation to write the explicit equation of the surface defined by equation \eqref{eq:sup}. 

For simpler notations, from now on we consider the equation of the cubic curve $y=\mathcal{A}(x)$ to be written as 
$$y=Ax^3+Bx^2+Cx+D, \quad A=A_3,B=A_2,C=A_1,D=A_0.$$

We have

	\[
	\begin{split}
	\widetilde{\text{T}}_1&=C(x_0\alpha+x_1\alpha^q+x_2\alpha^{q^2})+C^q(x_0\alpha^q+x_1\alpha^{q^2}+x_2\alpha)+C^{q^2}(x_0\alpha^{q^2}+x_1\alpha+x_2\alpha^{q})\\
	&=x_0\text{T}(\alpha C)+x_1\text{T}(\alpha C^{q^2})+x_2\text{T}(\alpha C^q),
	\end{split}
	\]
	
	\[
	\begin{split}
	\widetilde{\text{T}}_2=&B(x_0\alpha+x_1\alpha^{q}+x_2\alpha^{q^2})^2+B^q(x_0\alpha^q+x_1\alpha^{q^2}+x_2\alpha)^{2}+B^{q^2}(x_0\alpha^{q^2}+x_1\alpha+x_2\alpha^{q})^{2}\\
	&=x_0^2\text{T}(B\alpha^2)+x_1^2\text{T}(B\alpha^{2q})+x_2^2\text{T}(B\alpha^{2q^2})+2x_0x_1\text{T}(B\alpha^{q+1})+2x_0x_2\text{T}(B\alpha^{q^2+1})\\
	&+2x_1x_2\text{T}(B\alpha^{q^2+q}),\\
	\end{split}
	\]
	
	\[
	\begin{split}
	\widetilde{\text{T}}_3=&A(x_0\alpha+x_1\alpha^{q}+x_2\alpha^{q^2})^3+A^q(x_0\alpha^q+x_1\alpha^{q^2}+x_2\alpha)^{3}+A^{q^2}(x_0\alpha^{q^2}+x_1\alpha+x_2\alpha^{q})^{3}\\
	&=x_0^3\text{T}(A\alpha^3)+x_1^3\text{T}(A\alpha^{3q})+x_2^3\text{T}(A\alpha^{3q^2})+3x_0^2x_1\text{T}(A\alpha^{q+2})+3x_0^2x_2\text{T}(A\alpha^{q^2+2})\\
	&+3x_1^2x_2\text{T}(A\alpha^{q^2+2q})+3x_0x_1^2\text{T}(A\alpha^{1+2q})+3x_0x_2^2\text{T}(A\alpha^{1+2q^{2}})+3x_1x_2^2\text{T}(A\alpha^{q+2q^{2}}),\\
	\end{split}
	\]	
	
	\[
	\begin{split}
	\widetilde{\text{N}}=&(x_0\alpha^{q^2}+x_1\alpha+x_2\alpha^{q})(x_0\alpha^q+x_1\alpha^{q^2}+x_2\alpha)(x_0\alpha+x_1\alpha^{q}+x_2\alpha^{q^2})\\
	&=(x_0^3+x_1^3+x_2^3)\text{N}(\alpha)+(x_0^2x_1+x_1^2x_2+x_2^2x_0)\text{T}(\alpha^{q+2})\\
	&+(x_0^2x_2+x_1^2x_0+x_2^2x_1)\text{T}(\alpha^{2q+1})+x_0x_1x_2(3\text{N}(\alpha)+\text{T}(\alpha^3)).\\
	\end{split}
	\]

Hence, we are now able to rewrite \eqref{eq:sup} as 

	\begin{equation}
	\label{eq:sup1}
	\begin{split}
	0=&-(x_0^3+x_1^3+x_2^3)\text{N}(\alpha)-(x_0^2x_1+x_1^2x_2+x_2^2x_0)\text{T}(\alpha^{q+2})-
	(x_0^2x_2+x_1^2x_0+x_2^2x_1)\text{T}(\alpha^{2q+1})\\&
	-x_0x_1x_2(3\text{N}(\alpha)+\text{T}(\alpha^3))
	+x_0^3\text{T}(A\alpha^3)+x_1^3\text{T}(A\alpha^{3q})+x_2^3\text{T}(A\alpha^{3q^2})\\&+3x_0^2x_1\text{T}(A\alpha^{q+2})+3x_0^2x_2\text{T}(A\alpha^{q^2+2})+3x_1^2x_2\text{T}(A\alpha^{q^2+2q})\\&+3x_0x_1^2\text{T}(A\alpha^{1+2q})+3x_0x_2^2\text{T}(A\alpha^{1+2q^{2}})
	+3x_1x_2^2\text{T}(A\alpha^{q+2q^{2}})\\&
	+x_0^2\text{T}(B\alpha^2)+x_1^2\text{T}(B\alpha^{2q})+x_2^2\text{T}(B\alpha^{2q^2})+2x_0x_1\text{T}(B\alpha^{q+1})+2x_0x_2\text{T}(B\alpha^{q^2+1})\\
	&+2x_1x_2\text{T}(B\alpha^{q^2+q})+x_0\text{T}(\alpha C)+x_1\text{T}(\alpha C^{q^2})+x_2\text{T}(\alpha C^q)+E.
	\end{split}
	\end{equation} 
Let $\S_1$ be the surface defined by equation \eqref{eq:sup1}, and observe that it is defined over the field $\F_q$. We will denote with $\S_1(\F_q)$ the set of its affine $\F_q$-rational points.

\begin{remark}
\label{rem:intersection}
By construction, $\F_q$-rational points of $\S_1$ correspond to the intersections in $\mathbb{A}^2(\F_{q^3})$ between $\mathcal{N}_3$ and the curve $y=Ax^3+Bx^2+Cx+D$. 
In other words, our algebraic manipulations proved that there exists $x\in\F_{q^3}$ such that $\mathrm{N}(x)=\mathrm{T}(Ax^3+Bx^2+Cx+D)$ if and only if exists $(x_0,x_1,x_2)\in(\F_q)^3$ satisfying \eqref{eq:sup1} and its representation on the chosen normal basis is $x=x_0\alpha+x_1\alpha^q+x_2\alpha^{q^2}$. 
\end{remark}

Note that equation \eqref{eq:sup1} can be also rewritten as
	\begin{equation*}
	\begin{split}
	0=&-(x_0\alpha+x_1\alpha^q+x_2\alpha^{q^2})(x_0\alpha^q+x_1\alpha^{q^2}+x_2\alpha)(x_0\alpha^{q^2}+x_1\alpha+x_2\alpha^q)\\&
	+A(x_0\alpha+x_1\alpha^q+x_2\alpha^{q^2})^3+A^q(x_0\alpha^q+x_1\alpha^{q^2}+x_2\alpha)^{3}+A^{q^2}(x_0\alpha^{q^2}+x_1\alpha+x_2\alpha^q)^{3}\\&
	+B(x_0\alpha+x_1\alpha^q+x_2\alpha^{q^2})^2+B^q(x_0\alpha^q+x_1\alpha^{q^2}+x_2\alpha)^{2}+B^{q^2}(x_0\alpha^{q^2}+x_1\alpha+x_2\alpha^q)^{2}\\&
	+C(x_0\alpha+x_1\alpha^q+x_2\alpha^{q^2})+C^q(x_0\alpha^q+x_1\alpha^{q^2}+x_2\alpha)+C^{q^2}(x_0\alpha^{q^2}+x_1\alpha+x_2\alpha^q)+E.
	\end{split}
	\end{equation*}

Consider now the affine change of variables in $\mathbb{A}^3(\overline{\F}_q)$ defined by 
$$\psi(x_0,x_1,x_2)=M(x_0,x_1,x_2)^t=(X_0,X_1,X_2)^t$$ where $M$ is the non-singular matrix
	\[
	M=\begin{pmatrix}
	\alpha&\alpha^q&\alpha^{q^2}\\
	\alpha^q&\alpha^{q^2}&\alpha\\
	\alpha^{q^2}&\alpha&\alpha^q
	\end{pmatrix}
	\]  
	and let $\S_2$ be the surface obtained from $\S_1$ through $\psi$ (i.e. $\S_2=\psi(\S_1)$).
\begin{proposition}
\label{rem:correspondence}
$\S_2$ is a surface defined over $\F_{q^3}$, has equation
	\begin{equation}\notag
	\label{eq:sup2}
	X_0X_1X_2=AX_0^3+A^qX_1^3+A^{q^2}X_2^3+BX_0^2+B^qX_1^2+B^{q^2}X_2^2+CX_0+C^qX_1+C^{q^2}X_2+E
	\end{equation}
	and preserves the multiplicities of the points and of the components of $\S_1$. Moreover, the points on $\S_2$ of the form $(\beta,\beta^q,\beta^{q^2})$, $\beta\in\F_{q^3}$, are in bijection with the $\F_q$-rational points on $\S_1$.
\end{proposition}
\begin{proof}
These properties come from straightforward computations combined with the fact that $M$ is a non-singular affine transformation.
\end{proof}

In order to estimate the intersection between $\mathcal{N}_3$ and $y=Ax^3+Bx^2+Cx+D$, our goal is the determination of an upper bound for the number of (affine) $\F_q$-rational points of $\mathcal{S}_1$.

The final goal of this paper is to prove the following theorem:

\begin{theorem}
\label{thm:goal}
Consider the $\F_q$-rational cubic surface $\S_1$ associated to the intersections between $\mathcal{N}_3$ and $y=Ax^3+Bx^2+Cx+D$, $A\neq 0$. If $\S_1$ is absolutely irreducible and it is not a cone over an irreducible smooth plane cubic curve, then
		\[
		|\S_1(\F_q)|\le q^2+7q+1.
		\]
\end{theorem}

\begin{remark}
Note that the results in \cite{BS2018} prove Theorem \ref{thm:goal} under the assumption $A=0$. 
\end{remark}

Considering the discussion in section \ref{sec:general}, to prove Theorem \ref{thm:goal} we are left with some of the cases in which $\S_1$ has only three or four singular points to treat, so from now on we focus on these cases. However, we first treat separately the particular case $B=C=0$. 
Studying this case is interesting since it gives explicit information on the reducibility of $\S_1$, depending on its coefficients and the base field. Moreover, in the case $\mbox{char}(\mathbb{F}_{q})=3$, we find explicitly the form of the singular points that $\S_1$ can have, which is not possible for the general case.

\section{$\S_1$ irreducible with isolated singularities: case $B=C=0$}
\label{B=C=0}

In this section we will take in account the special case in which $\S_1$ is irreducible, has only isolated singularities, and $B=C=0$.

We investigate the singular points of $\S_2$, which is equivalent to studying singular points of $\S_1$, thanks to the map $\psi$ defined in precedence.
For this case of study, the equation of $\S_2$ is 
$$X_{0}X_{1}X_{2}= AX_{0}^{3}+A^{q}X_{1}^{3}+A^{q^{2}}X_{2}^{3}+E$$
and its affine singular points are the solutions to the following system:

\begin{equation}
\label{initial:system}
\begin{cases}
X_{0}X_{1}X_{2}= AX_{0}^{3}+A^{q}X_{1}^{3}+A^{q^{2}}X_{2}^{3}+E\\
X_{1}X_{2}=3AX_{0}^{2}\\
X_{0}X_{2}=3A^{q}X_{1}^{2}\\
X_{0}X_{1}=3A^{q^{2}}X_{2}^{2}\\
\end{cases}
\end{equation} 

From this system of equations, it is immediately clear that $E=0$ if and only if $(0,0,0)$ is a solution.

We distinguish now three different cases, depending on the characteristic of the field $\F_q$.

\begin{proposition}
Let $B=C=0$, and $\mathrm{char}(\mathbb{F}_{q})=3$. Then the only possible singularities of $\S_2$ are the following
\begin{itemize}
\item $(0,0,0)$ if and only if $E=0$;
\item $\left(0,0,-\frac{\varepsilon}{A^{\frac{q^{2}}{3}}}\right)$, $\left(0,-\frac{\varepsilon}{A^{\frac{q}{3}}},0\right)$ and $\left(-\frac{\varepsilon}{A^{\frac{1}{3}}},0,0\right)$ if $E\neq 0$ and $\varepsilon$ solution of $X_{2}^{3}=-\frac{E}{A^{q^{2}}}$.
\end{itemize}
\end{proposition}
\begin{proof}
In this case, system \eqref{initial:system} becomes 

\[
\begin{cases}
X_{0}X_{1}X_{2}= AX_{0}^{3}+A^{q}X_{1}^{3}+A^{q^{2}}X_{2}^{3}+E\\
X_{1}X_{2}=0\\
X_{0}X_{2}=0\\
X_{0}X_{1}=0\\
\end{cases}
\]
From the last three equations it follows that at least two among $X_{0}$, $X_{1}$ and $X_{2}$ have to be 0.
Suppose $X_{0}=X_{1}=0$, obtaining $A^{q^{2}}X_{2}^{3}+E=0.$ If $E=0$, the unique solution is $X_0=X_1=X_{2}=0$, otherwise we have 
$$X_{2}^{3}=-\frac{E}{A^{q^{2}}}$$
and this equation is soluble if and only if $-\frac{E}{A^{q^{2}}}$ is a cube in $\mathbb{F}_{q^{3}}$. Since $\mathrm{char}(\mathbb{F}_{q})=3$, $A^{q^{2}}$ is always a cube in $\mathbb{F}_{q^{3}}$ 
and $E$ is always a cube since $x^3$ is a permutation of the field. This means that the above equation is always soluble and there is one possible value for $X_{2}$
$$X_{2}=-\frac{\varepsilon}{A^{\frac{q^{2}}{3}}}$$
where $E=\varepsilon^{3}$. Hence, we have the solution $\left(0,0,-\frac{\varepsilon}{A^{\frac{q^{2}}{3}}}\right).$

Iterating the above reasoning for the other two cases, we have the thesis. It is also immediate to see that there are no singular points at the infinity. 
\end{proof}

\begin{proposition}
Let $B=C=0$, and $\mathrm{char}(\mathbb{F}_{q})=2$. Then the only possible singularity of $\S_2$ is the point $(0,0,0)$.
\end{proposition}
\begin{proof}

In this case, system \eqref{initial:system} becomes

\begin{equation}
\label{char2:system}
\begin{cases}
X_{0}X_{1}X_{2}= AX_{0}^{3}+A^{q}X_{1}^{3}+A^{q^{2}}X_{2}^{3}+E\\
X_{1}X_{2}=AX_{0}^{2}\\
X_{0}X_{2}=A^{q}X_{1}^{2}\\
X_{0}X_{1}=A^{q^{2}}X_{2}^{2}\\
\end{cases}
\end{equation}

Substituting, it is possible to see that $E$ must be equal to zero and that 
\eqref{char2:system} 
is equivalent to
\begin{equation}
\label{char2:div}
\begin{cases}
X_0^{3}=A^{q-1}X_1^3\\
X_1^{3}=A^{q^{2}-q}X_2^3\\
X_0^{3}=A^{q^2-1}X_2^3\\
\end{cases}
\end{equation}

Since we are looking for solutions with coordinates in $\overline{\F}_{q}$, note that we have all the solutions of the form
$$(\zeta_{3,i} A^{\frac{q^2-1}{3}}X_2,\zeta_{3,j} A^{\frac{q^2-q}{3}}X_2,X_2)$$
for $X_2\in \overline{\F}_q$ and $\zeta_{3,i}$, $\zeta_{3,j}$ cubic roots of unity.

Considering now the first equation in \eqref{char2:system}, we have that $(\zeta_{3,i} A^{\frac{q^2-1}{3}}X_2,\zeta_{3,j} A^{\frac{q^2-q}{3}}X_2,X_2)$ is a solution different from $(0,0,0)$ if and only if 
\begin{equation}\notag
\begin{split}
A^{\frac{q^2-1}{3}}A^{\frac{q^2-q}{3}}X_2^3&=A^{q^2}X_2^3+A^{q^2}X_2^3+A^{q^2}X_2^3\\
A^{\frac{q^2+q+1}{3}}&=1.
\end{split}
\end{equation} 
So, we have solutions if and only if $\textrm{N}(A)=1$.

We can sum up the situation as follows:

\begin{itemize}
\item if $\mathrm{N}(A)=1$, then there are more than 4 solutions to the system, in fact all triples of the form

$$(\zeta_{3,i} A^{\frac{q^2-1}{3}}X_2,\zeta_{3,j} A^{\frac{q^2-q}{3}}X_2,X_2)$$

satisfy the system, for every value of $X_2$ in $\overline{\F}_q$. However, this is not possible if the surface is irreducible (as it is in our case), due to Theorem \ref{thm:4singular}.

\item if $\mathrm{N}(A)\neq 1$, there are no solutions to the system different from $(0,0,0)$ (which is a solution if and only if $E=0$).
\end{itemize}

\end{proof}

Considering what we have just found regarding the solutions to system \eqref{char2:div}, we conclude also that there are no singular points at the infinity. In fact, the only solution to the system in our case would be the point with coordinates $[0:0:0:0]$.

The same result can be obtained for the general case, we do not show the proof of this proposition since it is analogue to the one just written. This can be easily adapted doing some different reductions in system \eqref{initial:system} and then noticing that $\mathrm{N}(A)=27$.
\begin{proposition}
Let $B=C=0$, and $\mathrm{char}(\mathbb{F}_{q})\ne2,3$. The only possible singularity of $\S_2$ is the point $(0,0,0)$.
\end{proposition}

\section{$\S_1$ irreducible with isolated singularities}
\label{sec:isosing:general}
We investigate now the general case $\S_1$ irreducible and with only isolated singularities.
As already noted in section \ref{subsec:isolatedsing}, if $\S_1$ is an irreducible cubic surface with isolated singularities, then its singular points are double points.

As pointed out above (Theorem \ref{thm:goal}), we wish to find a bound of type 
$$q^2+\eta q+\mu \leq q^2+7q+1$$
for the four possible cases of singularities presented in Theorem \ref{thm:4singular} ($\delta=1,2,3,4$).

We investigate the possible cases switching from the surface $\S_2$ to the surface $\S_1$ and vice-versa, thanks to the correspondence established between the two surfaces in section \ref{2cubic}. 
Note that the cases in which $\S_1$ has one singular $\F_q$-rational point or two singular $\F_{q^2}$-rational points have already been treated, more in general, in section \ref{sec:general}. Therefore, we are left to treat only some of the cases in which the surface has three or four singular points (see Remark \ref{rem:cases1} and \ref{rem:cases2}).

Firstly, note that the affine singular points of $\mathcal{S}_2$ correspond to the solutions of 
	\begin{equation}\notag
	\label{system:singular}
	\begin{cases}
X_0X_1X_2=AX_0^3+A^qX_1^3+A^{q^2}X_2^3+BX_0^2+B^qX_1^2+B^{q^2}X_2^2+CX_0+C^qX_1+C^{q^2}X_2+E\\
X_1X_2=3AX_0^2+2BX_0+C\\
X_0X_2=3A^qX_1^2+2B^qX_1+C^q\\
X_0X_1=3A^{q^2}X_2^2+2B^{q^2}X_2+C^{q^2}
	\end{cases}	
	\end{equation}

\begin{remark}
Regarding the singular points of $\S_2$ at the infinity, we note that we can exploit the computations we did while investigating the case $B=C=0$. In fact, considering $\mathbb{P}^{3}(\overline{\F}_{q})$ and points in $\mathbb{P}^{3}(\overline{\F}_{q})$ as $[X_0:X_1:X_2:X_3]$, with $X_3=0$ equation of the plane at the infinity, it can be seen that singular points at the infinity of $\S_2$ are those satisfying the system:

\begin{equation}\notag
\begin{cases}
X_{0}X_{1}X_{2}= AX_{0}^{3}+A^{q}X_{1}^{3}+A^{q^{2}}X_{2}^{3}\\
X_{1}X_{2}=3AX_{0}^{2}\\
X_{0}X_{2}=3A^{q}X_{1}^{2}\\
X_{0}X_{1}=3A^{q^{2}}X_{2}^{2}\\
\end{cases}
\end{equation} 
which is exactly system \eqref{initial:system} (i.e. the system considered in the case $B=C=0$) with $E=0$. Note that the fact that our system is equivalent to system \eqref{initial:system} with $E=0$ does not mean that we are considering the case in which the equation of $\S_2$ has $E=0$, it is just a matter of algebraic computations.\\
Therefore, considering what we have found in section \ref{B=C=0} regarding the solutions to this system, we conclude that there are no points at the infinity. In fact, the only solution to the system in our case ($\S_2$ irreducible) would be the point with coordinates $[0:0:0:0]$, which is not a point of the projective space. 
 
\end{remark}

Recall now Remark \ref{rem:rational} and another important fact (\cite{BS2018}).

\begin{remark}
If a singular point of $\S_2$ is $\F_{q^6}$-rational, then the corresponding singularity of $\S_1$ will be $\F_{q^2}$-rational, considering the way in which we have defined the correspondence between $\S_1$ and $\S_2$ in section \ref{2cubic}.
\end{remark}

\subsection{Three singular points}
\label{subsec:3sing}

We investigate now the case in which $\S_1$ has exactly three singular points $P_1$, $P_2$ and $P_3$. Recalling Remark \ref{rem:rational}, we know that there are two possible situations that may happen: either one among $P_1$, $P_2$ and $P_3$ is $\F_q$-rational, or all of them are $\F_{q^3}$-rational and conjugates. If the first situation happens, the bound given by \eqref{bound:1singular:6} applies. Hence, we need to investigate the second possible situation.

We can make use of the results in \cite{BS2018} to solve this case. 
First of all, the following proposition, which directly follows from Bézout's theorem, tells us that the three points cannot be collinear.

\begin{proposition}
Let $\mathcal{C}$ be a cubic curve such that it has three double points. Then $\mathcal{C}$ is completely reducible and splits in the product of three lines, each passing through a pair of its singular points.
\end{proposition}

The next step is to reduce the problem of counting $\F_q$-rational points on $\S_1$ to the problem of counting points on a certain quadric surface. 

More specifically, considering the correspondence given by $\psi$ between $\F_q$-rational points of $\S_1$ and $\F_{q^3}$-rational points of $\S_2$ (see Remark \ref{rem:correspondence}), we will know this number by counting points of the form $(\alpha,\alpha^q,\alpha^{q^2})$ on a certain quadric surface.

We first use the change of coordinates given by the following proposition in order to get a new form for the equation of the cubic surface $\S_1$.

\begin{proposition}[\cite{BS2018}]
\label{prop:coordchange}
Let $\mathcal{S}$ be a cubic surface over $\P^3(\F_q)$, considered with projective coordinates $[r_0:r_1:r_2:T]$, and such that it has exactly three conjugates $\F_{q^3}$-rational double points, namely $P_1,P_2$ and $P_3$. Then $\mathcal{S}$ is projectively equivalent to the surface having affine equation, for certain $\beta,\gamma\in\F_{q^3}$
		\[
		r_0r_1r_2+\beta r_0r_1+\beta^qr_1r_2+\beta^{q^2}r_0r_2+\gamma r_0+\gamma^qr_1+\gamma^{q^2}r_2=0.
		\]
\end{proposition}

Thanks to Proposition \ref{prop:coordchange}, we get a new model for the surface $\S_1$, and after using a Cremona transform we get the equation of the quadric

\[
	\mathcal{Q}:\beta z_3+\beta^qz_1+\beta^{q^2}z_2+\gamma z_2z_3+\gamma^qz_1z_3+\gamma^{q^2}z_1z_2-1=0.
	\]
which is irreducible, see \cite{BS2018} for details.

By construction, the points on $\S_1$ are in bijection with the points on $\mathcal{Q}$ in the form $(\delta,\delta^q,\delta^{q^2})$, where $\delta\in\F_{q^3}$. 

Nevertheless, it is possible to show that such points (and hence the affine $\F_q$-rational points of $\S_1$) are in one-to-one correspondence with the $\F_q$-rational points of an irreducible $\F_q$-rational quadric surface (see \cite{BS2018} for details). Then, by well-known results on quadric surfaces (see \cite[Section 15.3]{hirsch:finite}), the following bound is established:
	\begin{equation}
	\label{bound:3singular}
	|\S_1(\F_q)|= q^2+\eta q+1,\quad\eta\in\{0,1,2\}.
	\end{equation}

\subsection{Four singular points}
\label{subsec:4sing}

The last case we have to deal with is the one in which $\S_1$ has exactly four singular points $P_1$, $P_2$, $P_3$ and $P_4$. Note that if one among $P_1$, $P_2$, $P_3$ and $P_4$ is $\F_q$-rational, then we already have the desired bound given by \eqref{bound:1singular:6}. If instead the four points are two couples of conjugates $\F_{q^2}$-rational points, then the bound is already established by \eqref{bound:2singular}.

Therefore, we are left with the case in which the four points $P_1$, $P_2$, $P_3$ and $P_4$ are $\F_{q^4}$-rational and conjugates. 
To deal with this case, we introduce a change of coordinates and exploit this to transform the equation of the cubic surface $\S_1$ in a way that will be useful for our later computations.

\begin{proposition}
\label{prop:coordchange4}
Let $\mathcal{S}$ be a cubic surface over $\P^3(\F_q)$, considered with projective coordinates $[x_0:x_1:x_2:x_3]$, and such that it has exactly four conjugates $\F_{q^4}$-rational double points, namely $P_1,P_2$, $P_3$ and $P_4$. Then $\mathcal{S}$ is projectively equivalent to the surface having affine equation
		\[
		x_0x_1x_2+\beta_0 x_0x_1+\beta_1 x_1x_2+\beta_2 x_0x_2=0.
		\]
for certain $\beta_0,\beta_1,\beta_2 \in\F_{q^4}$.
\end{proposition}

\begin{proof}
We start observing that no three points among $P_1,P_2$, $P_3$ and $P_4$ are collinear. In fact, suppose without loss of generality that $P_1,P_2$, $P_3$ are collinear: then a general plane through $P_1,P_2$, $P_3$ would meet the surface in a cubic curve with three collinear double points, which means that the only possibility is that the curve is union of a double line $\ell$ (the line joining the three double points) and another line $r$ (see \cite{coray:singular}). 
In this case, the whole line $\ell$ would be double on the surface, which is impossible since by hypothesis $\mathcal{S}$ has only isolated singularities.

Hence, we have that no three points among $P_1,P_2$, $P_3$ and $P_4$ are collinear and from this we can deduce that they are not coplanar. In fact, if they were on a same plane, then the intersection of this plane with the cubic surface would be a reducible cubic curve which should contain each line connecting two distinct points among $P_1,P_2$, $P_3$ and $P_4$, and this is clearly impossible.

Therefore, we can set a new projective frame as follows: consider the four planes defined by the four possible combinations of the points, that is

\begin{itemize}
\item the plane $\pi$ through $P_2$, $P_3$ and $P_4$ 
\item the plane $\rho$ through $P_1$, $P_2$ and $P_3$
\item the plane $\tau$ through $P_1$, $P_2$ and $P_4$
\item the plane $\kappa$ through $P_1$, $P_3$ and $P_4$.
\end{itemize} 
Note that, since the points $P_1,P_2$, $P_3$ and $P_4$ are conjugates, it follows that the four planes defined above are conjugates as well, in particular 
\[
\pi^q = \kappa, \quad \kappa^q = \tau, \quad \tau^q=\rho, \quad \rho^q=\pi.
\]
Let $U$ be a point that does not lie in any of the planes $\pi, \rho, \tau, \kappa$, then we have that $(P_1,P_2,P_3,P_4,U)$ is a projective frame.

With respect to this projective frame, homogeneous coordinates for the points $P_1,P_2,P_3,P_4$ and $U$ are exactly:
\[
P_1=[1:0:0:0], \quad P_2=[0:1:0:0], \quad P_3=[0:0:1:0], \quad P_4=[0:0:0:1]\quad U=[1:1:1:1]
\]
and projective coordinates in this new frame are $[x_0:x_1:x_2:x_3]$, where $x_0=0$ is the equation of $\pi$, $x_3=0$ is the equation of $\rho$, $x_2=0$ is the equation of $\tau$ and $x_1=0$ is the equation of $\kappa$.

Considering this projective frame, we can write the equation of the surface as 
\[
	x_0x_1x_2+ x_3(\alpha_0x_0^2+\alpha_1x_1^2+\alpha_2x_2^2+\beta_0x_0x_1+\beta_1x_1x_2+\beta_2x_0x_2)+x_3^2(\gamma_0x_0+\gamma_1x_1+\gamma_2x_2)=0
\]
where $\alpha_i,\beta_i,\gamma_i\in\F_{q^4}$ for $i\in\{0,1,2\}$, since, as we observed, the planes $\pi,\rho,\tau$ and $\kappa$ are conjugates.

Note also that in the equation there is not the constant term, since $P_4\in \mathcal{S}$, and there are no terms in $x_i^3$, $i=0,1,2$, since $P_1\in \mathcal{S}$. 

Moreover, since $P_1$ and $P_4$ are double points of the surface $\mathcal{S}$, the line $r$ through them is contained in $\mathcal{S}$. So, if we write the expression of a general point on this line, we have $P_{\lambda,\mu}=(\lambda:0:0:\mu)$. Substituting this expression in the equation of $\mathcal{S}$ we have

\begin{equation}
\begin{split}
\mu \alpha_0 \lambda^2 + \mu^2 \gamma_0 \lambda =0\\
\mu \lambda (\alpha_0 \lambda +\gamma_0 \mu)=0\\
\end{split}
\end{equation}
and this equation holds for every value of $\mu$ and $\lambda$.
This yields that both $\alpha_0$ and $\gamma_0$ must be zero.

Iterating this reasoning also for the lines through $P_2$ and $P_4$ and through $P_3$ and $P_4$, we get the following equation for the surface $\mathcal{S}$:

\begin{equation}
\label{eq:surf4singchange}
x_0x_1x_2+\beta_0 x_0x_1+\beta_1 x_1x_2+\beta_2 x_0x_2=0
\end{equation}
\end{proof}

We apply now Proposition \ref{prop:coordchange4} to obtain an affine equation in the form \eqref{eq:surf4singchange} for $\S_1$, namely  
\begin{equation}
\label{eq:S1change}
x_0x_1x_2+\beta_0 x_0x_1+\beta_1 x_1x_2+\beta_2 x_0x_2=0
\end{equation}
where $\beta_0,\beta_1,\beta_2 \in\F_{q^4}$.

Applying the Cremona transformation defined by 
\[
(x_0,x_1,x_2)\longmapsto \biggl(\frac{1}{x_0},\frac{1}{x_1},\frac{1}{x_2}\biggr)
\]
and calling $z_0=\frac{1}{x_0}$, $z_1=\frac{1}{x_1}$, $z_2=\frac{1}{x_2}$, we see that $\S_1$ is mapped into a plane. To see this it is sufficient to take equation \eqref{eq:S1change} and divide it by $x_0 x_1 x_2$, obtaining

\begin{equation}
\label{eq:S1cremona}
1+\beta_0 \frac{1}{x_2}+\beta_1 \frac{1}{x_0}+\beta_2 \frac{1}{x_1}=0;
\end{equation}
then, applying the transformation gives as result
\begin{equation}
\label{eq:S1cremonaplane}
\Pi : 1+\beta_0 z_2+\beta_1 z_0+\beta_2 z_1=0.
\end{equation}

We remember that our final goal is to bound the number of $\F_q$-rational points of $\S_1$, that is equivalent to counting the number of points on $\S_2$ having form $(\gamma,\gamma^q,\gamma^{q^2})$, $\gamma\in \F_{q^3}$. 
Hence, we count points in this form on the plane $\Pi$ since these points are in direct correspondence with the ones on $\S_2$ of the same form. 

Writing down $\gamma$ on the normal basis we have $\gamma=w_1 \alpha + w_2 \alpha^q + w_3 \alpha^{q^2}$. From \eqref{eq:S1cremonaplane}, taking $w_1, w_2$ and $w_3$ as a set of variables over $\F_q$, we get the following equation
\begin{equation}
\label{eq:cremonapoints}
\beta_0(w_1\alpha^{q^2}+w_2\alpha+w_3\alpha^q)+\beta_1(w_1\alpha+w_2\alpha^q+w_3\alpha^{q^2})+\beta_2(w_1\alpha^q+w_2\alpha^{q^2}+w_3\alpha)+1=0.
\end{equation} 

Direct computations show that \eqref{eq:cremonapoints} is the equation of an $\F_{q^{12}}$-rational plane $\Pi^{\prime}$:
\begin{equation}
\label{eq:finalplane}
\Pi^{\prime}: w_1(\beta_0\alpha^{q^2}+\beta_1\alpha+\beta_2\alpha^q)+w_2(\beta_0\alpha+\beta_1\alpha^q+\beta_2\alpha^{q^2})+w_3(\beta_0\alpha^q+\beta_1\alpha^{q^2}+\beta_2\alpha)+1=0.
\end{equation}
By construction, $\F_q$-rational points of $\Pi^{\prime}$ are in one-to-one correspondence with the points we are looking for on our cubic surface. Hence, we want to estimate the number of $\F_q$-rational triples $(w_1,w_2,w_3)$ that are solutions to equation \eqref{eq:finalplane}. 

Since we are looking for $\F_q$-rational solutions, we recall that a plane can have at most $q^2$ $\F_q$-rational points. This implies that 
\begin{equation}
\label{bound:4singular}
|S_1(\F_q)|\leq q^2
\end{equation}
which is consistent with the statement of Theorem \ref{thm:goal}.

Therefore, considering all the cases discussed in this and in the previous sections, we conclude that we have finally proved Theorem \ref{thm:goal}.

In this final part, we will give some final remarks on the remaining cases, i.e., the case of the surface being a cone over an irreducible smooth cubic plane curve or being reducible. Note that we have already discussed these two cases in general in section \ref{sec:general}.

\begin{remark}
If $\S_1$ is a cone over a smooth absolutely irreducible cubic plane curve, we believe that it could be possible to obtain a better bound, in the same form as Theorem \ref{thm:goal}, writing explicitly the equation of such cone. Moreover, due to the variety of possible different equations of $\S_1$, it is not even clear to us if there exist choices of parameters that realize this situation.
\end{remark}

\begin{remark}
The cases in which $\S_1$ is reducible appear to be rare, and it is also unclear if they happen only for particular values of $q$. Nevertheless, up to now, we are not able to characterize them uniquely. However, in section \ref{B=C=0} we studied a particular case for which we were able to obtain explicit information on the reducibility of $\S_1$, depending on its coefficients and the base field.
\end{remark}

\begin{problem}
Characterize the reducibility of $\S_1$, and when $\S_1$ is a cone, using only relations on its coefficients and the base field.
\end{problem}

\section{AG codes arising from Norm--Trace curves}
\label{sec:agcodes}	
We already know that $\mathcal{N}_{3}$ has $N=q^5$ $\F_{q^3}$-rational points in $\mathbb{A}^2(\F_{q^3})$, and that 
\[
\mathscr{L}_{\F_{q^3}}(3q^2P_\infty)=\langle \{1,x,x^2,x^3,y,y^2,xy\} \rangle.
\]

Considering now the evaluation map 
\[
\begin{split}
\mathrm{ev}: \,\, \mathscr{L}_{\F_{q^3}}(3q^2P_\infty)&\longrightarrow(\F_{q^3})^{q^5}\\
f=ay^2+bxy+cy+dx^3+ex^2+fx+g&\longmapsto (f(P_1),\dots,f(P_N))\\
\end{split}
\]
the associated one-point code will be $C_{\mathscr{L}}(D,3q^2P_{\infty})=\mathrm{ev}(\mathscr{L}_{\F_{q^3}}(3q^2P_\infty))$, where the divisor $D$ is the formal sum of all the $q^5$ rational affine points of $\mathcal{N}_{3}(\F_{q^3})$. The weight of a codeword associated to the evaluation of a function $f\in \mathscr{L}_{\F_{q^3}}(3q^2P_\infty)$ corresponds to 

\[\mathrm{w}(\mathrm{ev}(f))=|\mathcal{N}_{3}(\F_{q^3})|-|\{\mathcal{N}_{3}(\F_{q^3})\cap\{ay^2+bxy+cy+dx^3+ex^2+fx+g=0\}\}|.\]

Using the results obtained in the previous sections, we can give some bounds in a variety of cases.

\begin{itemize}
    \item If $a=b=d=0$ then we are in a case already studied in \cite{BS2018}. More specifically:
    
    \begin{enumerate}
		\item if $c=0$ then we have to consider the zeros of $ex^2+fx+g$ that are points of $\mathcal{N}_{3}(\F_{q^3})$.
		\begin{enumerate}
			\item If $e=f=g=0$ then $\mathrm{w}(\mathrm{ev}(f))= 0$;
			\item if $e=f=0$ and $g\ne0$ then $\mathrm{w}(\mathrm{ev}(f))=q^5$;
			\item if $e=0$ and $f\ne0$ then $\mathrm{w}(\mathrm{ev}(f))= q^5-q^2$;
			\item if $f\ne0$ and $f^2-4eg=0$ then $\mathrm{w}(\mathrm{ev}(f))= q^5-q^2$;
			\item otherwise $\mathrm{w}(\mathrm{ev}(f))= q^5-2q^2$.
		\end{enumerate}
		\item On the other hand, if $c\ne0$ then we have to consider the points of $\mathcal{N}_{3}(\F_{q^3})$ that are zeros of $cy+ex^2+fx+g$.
		\begin{enumerate}
			\item If $e=f=g=0$ then $\mathrm{w}(\mathrm{ev}(f))= q^5-1$;
			\item if $e=f=0$ and $g\ne0$ then $\mathrm{w}(\mathrm{ev}(f))=q^5-q^2$;
			\item if $e=0$ and $f\ne0$ then, applying Bézout theorem, we have that 
			$$\mathrm{w}(\mathrm{ev}(f))~\ge~q^5~-~(q^2~+~q~+~1);$$
			\item otherwise, from what we said previously, $\mathrm{w}(\mathrm{ev}(f))\ge q^5-(q^2+7q+1)$.
		\end{enumerate}
		
	\end{enumerate}

    \item If $a=b=0$ and $d\ne0$ then we can obtain some information from our results on intersections.  
    \begin{enumerate} 
    \item If $c=0$ then we have to consider the points of $\mathcal{N}_{3}(\F_{q^3})$ that are zeros of $dx^3+ex^2+fx+g$.
    		\begin{enumerate}
			\item If $e=f=g=0$ then $\mathrm{w}(\mathrm{ev}(f))= q^5-q^2$;
			\item if $e=f=0$ and $g\ne0$ then $\mathrm{w}(\mathrm{ev}(f))=q^5-q^2$;
			\item otherwise $\mathrm{w}(\mathrm{ev}(f))\geq q^5-3q^2$.
		\end{enumerate}
	\item If $c\neq 0$ then we have to consider the points of $\mathcal{N}_{3}(\F_{q^3})$ that are zeros of $cy+dx^3+ex^2+fx+g$. We can do this exploiting our results on $\F_q$-rational points of the cubic surface obtained from the intersection of $\mathcal{N}_{3}$ and the cubic curve we are considering. As noted, we have different bounds according to the different shapes the surface $\S_1$ assumes.
		\begin{enumerate}
			\item If $\S_1$ is absolutely irreducible and it is not a cone over an irreducible smooth plane cubic curve, then $\mathrm{w}(\mathrm{ev}(f))\ge q^5-(q^2+7q+1)$;
			\item if there exist coefficients $c,d,e,f,g$ for which $\S_1$ is a cone over a smooth absolutely irreducible cubic plane curve, then $\mathrm{w}(\mathrm{ev}(f))\ge q^5-(q^2+2q\sqrt{q}+1)$;
			\item if there exist coefficients $c,d,e,f,g$ for which $\S_1$ is reducible, then $\mathrm{w}(\mathrm{ev}(f))\ge q^5-3q^2$.  	
		\end{enumerate}
    \end{enumerate}
    
  \item If $a\ne0$ or $b\ne 0$ unfortunately, from our results, we are not able to deduce information on the weights.

\end{itemize}

\begin{remark}
These considerations about the weight spectrum of the code tell us that, despite the fact that the dimension of the code increases with respect to the one studied in  \cite{BS2018}, the lower weights of the weight spectrum, which are those that contribute more for the computation of the PUE (Probability of the Undetected Error), do not seem to have many variations. 
\end{remark}

\begin{remark}
The results on the weights that we have obtained for the code $C_{\mathscr{L}}(D,3q^2P_{\infty})$ hold for a more general class of AG codes. 
We gave lower bounds on the weights of such a code considering some monomials that are comprised in the considered basis of the Riemann-Roch space $\mathscr{L}_{\F_{q^3}}(3q^2P_\infty)$. The same bounds hold for other codes on $\mathcal{N}_{3}$ associated to divisors $D$ and $G=kP_\infty$ with $k\ge 3q^2$. This follows because there exists a basis of $\mathscr{L}_{\F_{q^3}}(kP_\infty)$ (note that $\mathscr{L}_{\F_{q^3}}(3q^2P_\infty)$ $\subseteq \mathscr{L}_{\F_{q^3}}(kP_\infty)$ for $k\ge 3q^2$) that contains those monomials, when $k\ge 3q^2$. Hence, our results have impact on a vast range of codes arising from $\mathcal{N}_{3}$.   
\end{remark}

As noted above, our discussion does not cover all the possible cases since the basis of the Riemann-Roch space $\mathscr{L}_{\F_{q^3}}(3q^2P_\infty)$ has also the monomials $xy$ and $y^2$. 

Using our approach, it seems difficult to study the intersections of $\mathcal{N}_{3}$ with curves with terms in $xy$ or $y^2$, as the equation of the surface corresponding to the intersections would be much more complicated (for instance it can be a quartic surface).
Therefore, it still remains an open problem to determine the weight spectrum of the code $\mathscr{L}_{\F_{q^3}}(3q^2P_\infty)$.

\begin{problem}
Determine the weight spectrum of the code $\mathscr{L}_{\F_{q^3}}(3q^2P_\infty)$.
\end{problem}

\section*{Acknowledgments}
The results showed in this paper are included in L. Vicino’s MSc thesis (supervised by the first and the second author).


\bigskip


\begin{thebibliography}{99}

    \bibitem{abhyankar1960cubic} S. Abhyankar. Cubic surfaces with a double line. Memoirs of the College of Science, University of Kyoto. Series A: Mathematics {\bf 32} (3), 455--511 (1960).
    
	\bibitem{BalB}{E. Ballico, M. Bonini. On the weights of dual codes arising from the GK curve, Applicable Algebra in Engineering, Communication and Computing, accepted (arXiv:1909.08126).}
    
    \bibitem{ballico2013duals} E. Ballico, A. Ravagnani. On the duals of geometric Goppa codes from norm-trace curves. Finite Fields and Their Applications {\bf 20}, 30--39 (2013).
    
    \bibitem{ballico2014} E. Ballico, A. Ravagnani. On the geometry of Hermitian one-point codes. Journal of Algebra {\bf 397}, 499--514 (2014).
    
    \bibitem{BB2018}{D. Bartoli, M. Bonini. Minimum weight codewords in dual algebraic-geometric codes from the Giulietti-Korchm\'aros curve. Des. Codes Cryptography, to appear (https://doi.org/10.1007/s10623-018-0541-y) (2018).}

    \bibitem{BMZ1} D. Bartoli, M. Montanucci, G. Zini. Multi point AG codes on the GK maximal curve. Des. Codes and Cryptography {\bf 86} (1), 161-177 (2018).
	
	\bibitem{BBZ2020}{M. Bonini, M. Montanucci, G. Zini. On plane curves given by separated polynomials and their automorphisms, Advances in Geometry, 20(1), 61--70 (2020).}
	
	\bibitem{BS2018} M. Bonini, M. Sala. Intersections between the norm-trace curve and some low degree curves. Finite Fields and Their Applications, {\bf 67}, 101715 (2020).
	
	\bibitem{bruce1979classification} J.W. Bruce, C.T.C. Wall. On the classification of cubic surfaces. Journal of the London Mathematical Society {\bf 2} (2), 245--256 (1979).
	
	\bibitem{Cooley2014} J.A. Cooley. Cubic Surfaces over Finite Fields. Ph.D. Thesis (2014).
	
	\bibitem{coray:singular} D.F. Coray, M.A. Tsfasman. Arithmetic on singular Del Pezzo surfaces. Proceedings of the London Mathematical Society {\bf 3} (1), 25--87 (1988).
	
	\bibitem{C2012}{A. Couvreur. The dual minimum distance of arbitrary-dimensional algebraic–geometric codes. Journal of Algebra \textbf{350}, 84-107 (2012).}
	
	
	\bibitem{datta:sixlines} M. Datta.  Maximum number of $\mathbb{F}_{q}$-rational points on nonsingular threefolds in $\mathbb{P}^{4}$, Finite Fields and Their Applications {\bf 59},  86--96 (2019).
	
	\bibitem{donati2009intersection} G. Donati, N. Durante, G. Korchm{\'a}ros. On the intersection pattern of a unital and an oval in PG(2, $q^{2}$). Finite Fields and Their Applications {\bf 15} (6), 785--795 (2009).
	
	\bibitem{elkies2006linear} N.D. Elkies. Linear codes and algebraic geometry in higher dimensions. Preprint (2006).
	
	\bibitem{FMT2013}J.I. Farrán, C. Munuera, G. C. Tizziotti, F. Torres. Gröbner basis for norm-trace codes. Journal of Symbolic Computation {\bf 48}, 54--63 (2013).
	
	\bibitem{geil2003codes} O. Geil. On codes from norm--trace curves. Finite Fields and Their Applications {\bf 9} (3), 351--371 (2003).
    
    \bibitem{goppa1981codes} V.D. Goppa. Codes on algebraic curves. Soviet Math. Dokl. {\bf 24}, 170--172 (1981).
    
    \bibitem{Goppa82} V.D. Goppa. Algebraic-geometric codes. Izv. Akad. NAUK SSSR {\bf 46}, 75--91 (1982).
    
    \bibitem{hartshorne1997families} R. Hartshorne. Families of Curves in $\mathbb{P}^3$ and Zeuthen's Problem. American Mathematical Society (1997).
    
    \bibitem{hirsch:projective} J.W.P. Hirschfeld. Projective Geometries Over Finite Fields. Clarendon Press (1998).  
	
	\bibitem{hirsch:finite} J.W.P. Hirschfeld. Finite Projective Spaces of Three Dimensions. Clarendon Press (1985).
    
    \bibitem{korch:alg} J.W.P. Hirschfeld, G. Korchmáros, F. Torres. Algebraic Curves over a Finite Field. Princeton University Press (2008).
    
    \bibitem{kaplan2013rational} N. Kaplan. Rational Point Counts for del Pezzo Surfaces over Finite Fields and Coding Theory. Ph.D. Thesis (2013).
    
    \bibitem{Klove}  T. Klove. Codes for error detection.  Series on Coding Theory and Cryptology, 2. World Scientific Publishing Co. Pte. Ltd., Hackensack (2007).
	
	\bibitem{lidl:ff} R. Lidl, H. Niederreiter. Introduction to finite fields and their applications. Cambridge University Press (1986).
		
	\bibitem{manin:cubic} Yu.I. Manin, Cubic forms: algebra, geometry, arithmetic. Vol. 4. Elsevier (1986).
		
	\bibitem{MPS2014} C. Marcolla, M. Pellegrini, M. Sala,  On the Hermitian curve and its intersection with some conics, Finite Fields and Their Applications {\bf 28}, 166--187 (2014).
		
	\bibitem{MPS2016} C. Marcolla, M. Pellegrini, M. Sala.  On the small-weight codewords of some Hermitian codes. J. Symbolic Comput. {\bf 73},  27--45 (2016).
	
	\bibitem{M2004} G.L. Matthews. Codes from the Suzuki function field. IEEE Transactions on Information Theory {\bf 50} (12), 3298-3302 (2004).
	
	\bibitem{MTT2008}C. Munuera, G. C. Tizziotti, F. Torres. Two-point codes on Norm-Trace curves. Coding Theory and Applications. Springer, Berlin, Heidelberg. 128--136 (2008).
	
	\bibitem{nakai1950note} Y. Nakai. Note on the intersection of an algebraic variety with the generic hyperplane. Memoirs of the College of Science, University of Kyoto. Series A: Mathematics {\bf 26} (2), 185--187 (1950).
	
	\bibitem{roczen1996cubic} M. Roczen. Cubic surfaces with double points in positive characteristic. Algebraic Geometry and Singularities. Springer, 375--382 (1996).
	
	\bibitem{segre:ts} B. Segre. A note on arithmetical properties of cubic surfaces. Journal of the London Mathematical Society {\bf 1} (1), 24--31 (1943).
	
	\bibitem{Stichtenoth1988} H. Stichtenoth.  A note on Hermitian codes over $GF(q^2)$. IEEE Trans. Inf. Theory {\bf 34}(5), 1345--1348 (1988).
	
	\bibitem{stich:agcodes} H. Stichtenoth. Algebraic Function Fields and Codes. Springer (2009).
	
	\bibitem{Tiersma1987} H.J. Tiersma. Remarks on codes from Hermitian curves. IEEE Trans. Inf. Theory {\bf 33}(4), 605--609 (1987).
	
	\bibitem{weil1958abstract} A. Weil. Abstract versus classical algebraic geometry. Matematika {\bf 2} (4), 59--66 (1958).

    \bibitem{XL2000} C.P. Xing, S. Ling. A class of linear codes with good parameters from algebraic curves. IEEE Trans. Inf. Theory {\bf 46}(4), 1527--1532 (2000).


	\end{thebibliography}
\end{document}